\documentclass[a4paper, 11pt]{amsart}

\usepackage[latin1]{inputenc}
\usepackage{subfigure}
\usepackage{psfrag, amscd, amssymb, amsmath}
\usepackage{graphics}
\usepackage{pgf,pgfarrows,pgfnodes,pgfautomata,pgfheaps,pgfshade}
\usepackage{hyperref}
\usepackage{pstricks}
\usepackage{epsfig}
\usepackage{pst-grad} 
\usepackage{pst-plot} 
\usepackage{color}

\newtheorem{lem}{Lemma}[section]
\newtheorem{thm}[lem]{Theorem}

\begin{document}
\title{Hankel determinants of sums of consecutive weighted Schr\"{o}der numbers}

\author{Sen-Peng Eu}
\address{Department of Applied Mathematics\\
National University of Kaohsiung\\
Kaohsiung 811, Taiwan, ROC} \email{speu@nuk.edu.tw}

\author{Tsai-Lien Wong}
\address{Department of Applied Mathematics\\
National  Sun Yat-sen University\\
Kaohsiung 804, Taiwan, ROC} \email{tlwong@math.nsysu.edu.tw}

\author{Pei-Lan Yen}
\address{Department of Applied Mathematics\\
National University of Kaohsiung\\
Kaohsiung 811, Taiwan, ROC} \email{yenpl.tw@gmail.com}

\thanks{Partially supported by National Science Council, Taiwan under grant NSC
98-2115-M-390-002-MY3, NSC 99-2115-M-110-001-MY3}




\maketitle

\begin{abstract}
 For a real number $t$, let $r_\ell(t)$ be the total weight of all $t$-large Schr\"{o}der paths of length $\ell$,
 and $s_\ell(t)$ be the total weight of all $t$-small Schr\"{o}der  paths of length $\ell$.
For constants $\alpha, \beta$, in this article we derive recurrence
formulae for the determinats of the Hankel matrices $\det_{1\le
i,j\le n} (\alpha r_{i+j-2}(t) +\beta r_{i+j-1}(t))$,
  $\det_{1\le i,j\le n} (\alpha r_{i+j-1}(t) +\beta r_{i+j}(t))$,
   $\det_{1\le i,j\le n} (\alpha s_{i+j-2}(t) +\beta s_{i+j-1}(t))$, and  $\det_{1\le i,j\le n} (\alpha s_{i+j-1}(t) +\beta s_{i+j}(t))$
combinatorially via suitable lattice path models.





\end{abstract}


\section{Introduction}

\subsection{Hankel determinants from Catalan, Motzkin, and Schr\"oder numbers}
Let $\{a_\ell\}_{\ell\ge 0}$ be a sequence. For a nonnegative
integer $k$, the \emph{Hankel matrix} $A^{(k)}_n$ of order $n$
generated by this sequence is defined to be the matrix
$$A_n^{(k)}=(a_{k+i+j-2})_{1 \leq i, j \leq n}.$$
%

When $\{a_n\}_{n\ge 0}$ is one of the three classic combinatorial
sequences (Catalan, Motzkin, or Schr\"oder numbers) arising from the
lattice path enumerations, the problem to compute the determinant
$\det(A_n^{(k)})$ has been extensively studied. Readers may refer to
\cite{MA99, ITZ09, Kra05, Kra10} for more examples, especially the
comprehensive references listed in~\cite{Kra10}.

We give a quick introduction. The Catalan number
$c_\ell=\frac{1}{\ell+1}{2\ell\choose \ell}$ counts the number of
Dyck paths of length $\ell$, which are the lattice paths in the
plane $\mathbb{Z} \times \mathbb{Z}$ from $(0, 0)$ to $(2\ell, 0)$
using steps $\textsf{U}=(1,1)$, $\textsf{D}=(1,-1)$ that never pass
below the $x$-axis. It is a folklore that $\det_{1\le i,j\le
n}(c_{i+j-2})=1$, $\det_{1\le i,j\le n}(c_{i+j-1})=1$ and
$\det_{1\le i,j\le n}(c_{i+j})=n+1$. In 1986 De Sainte-Catherine and
Viennot~\cite{CV85} proved that $\det_{1\le i,j\le
n}(c_{i+j+k-2})=\prod_{1\le i\le j\le k-1 }\frac{i+j+2n}{i+j}$. A
very extensive generalization is given recently by Krattenthaler
in~\cite{Kra10}.

The Motzkin numbers $\{m_\ell\}_{\ell\ge 0}=\{1,1,2,4,9,21,51,\dots
\}$ count the number of Motzkin paths of length $\ell$, which are
the lattice paths in the plane $\mathbb{Z} \times \mathbb{Z}$ from
$(0, 0)$ to $(\ell, 0)$ using steps $\textsf{U}=(1,1),
\textsf{D}=(1,-1), \textsf{L}=(1,0)$ that never pass below the
$x$-axis. In 1998 Aigner~\cite{Aig98} proved that $\det_{1\le i,j\le
n}(m_{i+j-2})=1$ for all $n$ and $\det_{1\le i,j\le n}(m_{i+j-1})$
equals $1$ if $n\equiv 0,1 \bmod 6$, equals $0$ if $n\equiv 2,5
\bmod 6$, or equals $-1$ if $n\equiv 3,4 \bmod 6$.

The large Schr\"{o}der numbers $\{r_\ell\}_{\ell\ge
0}=\{1,2,6,22,90,394,1806,\dots \}$ count the number of large
Schr\"{o}der paths of length $\ell$, which are the paths in the
plane $\mathbb{Z} \times \mathbb{Z}$ from $(0, 0)$ to $(2\ell, 0)$
using $\textsf{U}=(1, 1)$, $\textsf{D}=(1, -1)$, $\textsf{L}=(2, 0)$
that never pass below the $x$-axis. And the small Schr\"{o}der
numbers $\{s_\ell\}_{\ell\ge 0}=\{1,1,3,11,45,197,903,\dots\}$ count
the number of small Schr\"{o}der paths of length $\ell$, which are
large Schr\"{o}der paths of length $\ell$ with no level steps on the
$x$-axis. By applying Gessel-Viennot-Lindstr\"{o}m lemma, in 2005 Eu
and Fu~\cite{EF05} proved that $\det_{1\le i,j\le
n}(r_{i+j-2})=2^{{n\choose 2}}$, $\det_{1\le i,j\le
n}(r_{i+j-1})=2^{{n+1\choose 2}}$, $\det_{1\le i,j\le
n}(s_{i+j-2})=2^{{n\choose 2}}$, and $\det_{1\le i,j\le
n}(s_{i+j-1})=2^{{n\choose 2}}$. At the same time Brualdi and
Kirkland also obtained the results in the cases of large
Schr\"{o}der numbers via linear algebra~\cite{BK05} .

Note that the determinants $\det(A_n^{(k)})$ can be obtained for all
$k\ge 2$ once we know $\det(A_n^{(0)})$ and $\det(A_n^{(1)})$. This
fact~\cite{Aig07} is from the well-known identity
\begin{equation}\label{dog}
   \det(A_{n+1}^{(k)})\det(A_{n-1}^{(k+2)})=\det(A_{n}^{(k)})\det(A_{n}^{(k+2)})-\det(A_{n}^{(k+1)})^2,
\end{equation}
for $n\ge 1$.

\subsection{Hankel determinants for sums of two consecutive terms}
A variation is to consider the determinant of the Hankel matrix
generated by the sequence $\{a_\ell+a_{\ell+1}\}_{\ell\ge 0}$. That
is, to consider the determinant $\det_{1 \leq i, j \leq n}
(a_{k+i+j-2}+a_{k+i+j-1})$.

For Catalan numbers, in 2002 Cvetkovi\'c, Rajkovi\'c and Ivkovi\'c
proved algebraically that $$\displaystyle \det_{1\le i,j \le n}
(c_{i+j-2}+c_{i+j-1})=f_{2n+1} \quad \mbox{and} \quad \det_{1\le i,j
\le n} (c_{i+j-1}+c_{i+j})=f_{2n+2},$$ where $f_n$ is the Fibonacci
number~\cite{CRI02}. This elegant result stimulated several
follow-up works, see~\cite{BCQY10, CF07, Cig02, CK11, Kra10, RPB07}
for examples.

The Motzkin case was also done by several authors~\cite{CY11, CK11}.
One can generalise to the weighted version. For a real number $t$, a
$t$-Motzkin path is a Motzkin path in which the steps
$\textsf{U,D,L}$ have weight $1,1,t$ respectively, and the weight of
a path is the product of the weights of its steps. Let $m_\ell(t)$
be the sum of weights of all $t$-Motzkin paths of length $\ell$,
then the Hankel determinants $\det_{1\le i,j \le n}
(m_{k+i+j-2}(t))$ and $\det_{1\le i,j \le n} (m_{k+i+j-1}(t))$ were
computed in ~\cite{Kra05,SX08} for examples. By using of the lattice
path arguments, Cameron and Yip~\cite{CY11} also obtained when
$k=0,1$ the recurrence formulae of the determinant
$$\det_{1\le i,j \le n} (m_{k+i+j-2}(t)+m_{k+i+j-1}(t)).$$


Similary, let a \emph{$t$-large (or $t$-small) Schr\"{o}der path} be
a large (or small) Schr\"{o}der path in which the steps
$\textsf{U,D,L}$ are weighted $1,1,t$ respectively, and the weight
of this path is the product of the weights of its steps. Let
$r_\ell(t)$ (or $s_\ell(t)$) denote the sum of weights of all
$t$-large (or $t$-small) Schr\"{o}der paths of length $\ell$. Note
that $r_0(t)=s_{0}(t)$ and $r_\ell(t)=(1+t)s_\ell(t)$ for $\ell \geq
1$. Recently Sulanke and Xin~\cite{SX08} proved that
$$\displaystyle \det_{1\le i,j \le n} (r_{i+j-2}(t))=(1+t)^{\binom{n}{2}} \qquad \mbox{and} \qquad
\det_{1\le i,j \le n} (r_{i+j-1}(t))=(1+t)^{\binom{n+1}{2}}.$$ Hence
it is natural to consider the determinants of the Hankel matrices
with entries the sum of weighted large or small Schr\"oder numbers.
In 2007, Rajkovi{\'c}, Petkovi{\'c}, and Barry \cite{RPB07} gave the
explicit formula
\begin{alignat*}{2}
    \det_{1\le i,j \le n} (r_{i+j-2}(t)+r_{i+j-1}(t))
        &=\frac{L^{{n\choose 2}}}{2^{n+1}\sqrt{L^2+4}}((\sqrt{L^2+4}+L)(\sqrt{L^2+4}+L+2)^n\\
        &\phantom{==========}+(\sqrt{L^2+4}-L)(L+2-\sqrt{L^2+4})^n),
\end{alignat*}
where $L=1+t$. Their proof was done algebraically by way of
orthogonal polynomials.

In this paper, we will compute combinatorially the Hankel
determinants with entries the linear combinations of two consecutive
terms of $t$-large(or $t$-small) Schr\"oder numbers.

\subsection{Main results}
Denote $H_n^{(k)}(t):=(r_{i+j+k-2}(t))_{1 \leq i, j \leq n}$ and
$G_n^{(k)}(t):=(s_{i+j+k-2}(t))_{1 \leq i, j \leq n}$. For constants
$\alpha, \beta$, we
define
 $$H_n^{(k, k+1)}(t):=(\alpha r_{k+i+j-2}(t)+\beta r_{k+i+j-1}(t))_{1\le i,j\le n},$$
 $$G_n^{(k, k+1)}(t):=(\alpha s_{k+i+j-2}(t)+\beta s_{k+i+j-1}(t))_{1\le i,j\le n}.$$

Our main results are the following recurrences of the Hankel
determinants for $k=0,1$.

\begin{thm}\label{th:L01}
We have the folloiwng recurrences.
\begin{enumerate}
\item Let $\Theta_0(t)=1$ and $\Theta_n(t)=\frac{1}{(1+t)^{n \choose
2}}\det(H_n^{(0, 1)}(t))$ for $n \geq 1$. Then
$$\Theta_n(t)=\alpha \sum_{m=0}^{n-1}\beta^m \Theta_{n-1-m}(t) +
\beta(1+t)\Theta_{n-1}(t).$$
%
%
\item
Let $\Phi_0(t)=1$ and $\Phi_n(t)=\frac{1}{(1+t)^{n+1 \choose
2}}\det(H_n^{(1,2)}(t))$ for $n \geq 1$. Then
$$\Phi_n(t)=\alpha \sum_{m=0}^{n-1}\beta^m \Phi_{n-1-m}(t) +
\beta(1+t)\Phi_{n-1}(t)+\beta^n.$$
%
%
\item
Let $\Psi_0(t)=1$ and $\Psi_n(t)=\frac{1}{(1+t)^{n \choose
2}}\det(G_n^{(0,1)}(t))$ for $n \geq 1$. Then
$$\Psi_n(t)=\alpha \sum_{m=0}^{n-1}\beta^m \Psi_{n-1-m}(t) +
\beta(1+t)\Psi_{n-1}(t)-t\beta^n.$$
%
%
\item
Let $\Gamma_0(t)=1$ and $\Gamma_n(t)=\frac{1}{(1+t)^{n \choose
2}}\det(G_n^{(1,2)}(t))$ for $n \geq 1$. Then
$$\Gamma_n(t)=\alpha \sum_{m=0}^{n-1}\beta^m \Gamma_{n-1-m}(t) +
\beta(1+t)\Gamma_{n-1}(t)+\beta^n.$$
\end{enumerate}
\end{thm}

We are happy to stay in the recurrences as the exact formulae are
usually messy even when they are not hard to derive. For example,
see $\det_{1\le i,j \le n} (r_{i+j-2}(t)+r_{i+j-1}(t))$ just
mentioned in the last paragraph.

\medskip

Note that if specializing to $t=0, \alpha=\beta=1$ in Theorem
\ref{th:L01} (1) and (2), we obtain the result in~\cite{CRI02}. If
letting $\alpha=1, \beta=0$ in Theorem \ref{th:L01} (1) and (2), we
obtain the mentioned result in~\cite{SX08}. If letting
$\alpha=\beta=1$ in Theorem \ref{th:L01} (1), with $\Theta_0(t)=1$
we have for $n\ge 1$ the nice recurrence
$$\Theta_n(t)=(t+2)\Theta_{n-1}(t)+\Theta_{n-2}(t)+\dots +\Theta_1(t)+\Theta_0(t)$$
and the main result in~\cite{RPB07} can be easily recovered.
Similarly by letting $\alpha=\beta=1$ and $\Phi_0(t)=1$ in Theorem
\ref{th:L01} (2) we have the recurrence
$$\Phi_n(t)=(t+2)\Phi_{n-1}(t)+\Phi_{n-2}(t)+\dots +\Phi_1(t)+\Phi_0(t)+1$$
for $n\ge 1$. If further letting $t=1$ we have
$\{\Theta_n(1)\}_{n\ge 0}=1,3,10,34,116,\dots$ and
$\{\Phi_n(1)\}_{n\ge 0}=1,4,14,48,164,\dots$, but now we are dealing
with $\det_{1\le i,j\le n} (r_{i+j-2}+r_{i+j-1})$ and $\det_{1\le
i,j\le n} (r_{i+j-1}+r_{i+j})$. It can be checked that
$\Theta_3(1)=34$ and $\Phi_3(1)=48$ agree with the determinants
$$\frac{1}{2^3}\det\left(\begin{array}{ccc}
    1+2 & 2+6 & 6+22 \\
    2+6 & 6+22 & 22+90 \\
    6+22 & 22+90 & 90+394
  \end{array}\right)=34
\quad \mbox{and} \quad \frac{1}{2^6}\det\left(\begin{array}{ccc}
    2+6 & 6+22 & 22+90 \\
    6+22 & 22+90 & 90+394 \\
    22+90 & 90+394 & 394+1806
  \end{array}\right)=48.
$$

\medskip

We will prove these results combinatorially by applying the
Lindstr\"om-Gessel-Viennot lemma on suitable lattice paths models. A
simplified version serving our need will be introduced in the next
section. Readers can refer to ~\cite{Aig07, GV85, Kra06} for more
information.

Here we would like to make some points about the proofs. The proofs
are unusual in the sense that from a conceptual viewpoint,
Theorem~\ref{th:L01} (1), (2), (3) are proved \emph{simultaneously},
while Theorem~\ref{th:L01} (4) is merely a direct corollary of (2).
The reason is that in order to obtain the results on $t$-large
Schr\"oder numbers one needs the corresponding results on $t$-small
Schr\"oder numbers (of smaller size) and vice versa. These
`intertwined' facts reflect in the two lemmas (Key Lemma I and Key
Lemma II) in Section~\ref{se:KL} and two lemmas in
Section~\ref{sec4}.

The rest of this paper is organized as follows. In
Section~\ref{sec2} we introduce the combinatorial models. In Section
~\ref{se:KL} we prove Key Lemma I and Key Lemma II. After more
intermediate results in Section~\ref{sec4} and Section~\ref{sec5},
we complete the proofs in Section~\ref{sec6}.


\section{Lattice path models}\label{sec2}
For each $0 \leq i \leq n$, denote
$A_{n,\overline{i}}^{(k)}$ the $n\times n$ matrix that results from
$A_{n+1}^{(k)}$ by deleting the $(n+1)$-th row and the $(i+1)$-th
column, that is,

$$A_{n, \overline{i}}^{(k)}:=\begin{pmatrix}
                         a_{k} & a_{k+1} & \cdots & a_{k+i-1} & a_{k+i+1}& \cdots & a_{k+n} \\
                         a_{k+1} & a_{k+2} & \cdots & a_{k+i} & a_{k+i+2}& \cdots & a_{k+n+1} \\
                         \vdots & \vdots & \ddots & \vdots & \vdots & \ddots & \vdots \\
                         a_{k+n-1} & a_{k+n} & \cdots & a_{k+i+n-2} & a_{k+i+n}& \cdots & a_{k+2n-1} \\
                       \end{pmatrix}_{n \times n}.$$
Recall that $H_n^{(k)}(t):=(r_{i+j+k-2}(t))_{1 \leq i, j \leq n}$
and $G_n^{(k)}(t):=(s_{i+j+k-2}(t))_{1 \leq i, j \leq n}$ and define
$H_{n, \overline{i}}^{(k)}(t)$ and $G_{n, \overline{i}}^{(k)}(t)$
accordingly.

\subsection{Lattice path models}
Define the directed graph $G$ with the vertex set $\{(x,y)\in
\mathbb{Z}^2: y\ge 0\}$ and the edge set
$\{(i,j)\to(i+2,j)\}\cup\{(i,j) \to (i+1,j+1)\}\cup\{ (i,j) \to
(i+1,j-1)\}$ for all legal $(a,b)$, each edge pointing to the right
and each level step being of weight $t$. Then a $t$-large (or small)
Schr\"oder path is a directed path on $G$ which starts from and ends
at the $x$-axis. Now we introduce our lattice path models.

\begin{enumerate}
\item
 Let $\Pi_n^{(k)}(t)$ (resp. $\Omega_n^{(k)}(t)$) be the set of
$n$-tuples $(\pi_k, \pi_{k+1}, \pi_{k+2},\cdots, \pi_{k+n-1})$
of $t$-large (resp. $t$-small) Schr\"{o}der paths subject to the
following two conditions (See Fig~\ref{fig:H3}.):
\begin{itemize}
  \item[$\bullet$] The path $\pi_{k+j}$ goes from $(-k-2j, 0)$ to $(k+2j, 0)$, $0 \leq j \leq n-1$.
  \item[$\bullet$] Any two paths do not intersect.
\end{itemize}

    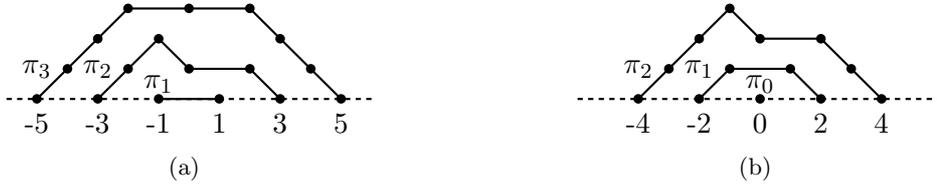
\begin{figure}[ht]
       \centering
        \subfigure[]{
            \begin{pspicture}(-3,-.5)(3,1.5)
                \psset{unit=.4cm}
                \psline(-5,0)(-2,3)(2,3)(5,0)
                \psdots(-5,0)(-4,1)(-3,2)(-2,3)(0,3)(2,3)(4,1)(3,2)(5,0)
                \psline(-3,0)(-1,2)(0,1)(2,1)(3,0)
                \psdots(-3,0)(-2,1)(-1,2)(0,1)(2,1)(3,0)
                \psline(-1,0)(1,0)
                \psdots(-1,0)(1,0)
                \psline[linestyle=dashed,dash=2pt 2pt](-6,0)(6,0)
                \rput(-5,1){$\pi_3$}
                \rput(-3,1){$\pi_2$}
                \rput(-1,.5){$\pi_1$}
                \rput(-5,-.8){-5}
                \rput(-3,-.8){-3}
                \rput(-1,-.8){-1}
                \rput(1,-.8){1}
                \rput(3,-.8){3}
                \rput(5,-.8){5}
            \end{pspicture}
          }
        \hspace{1cm}
    \subfigure[]{
            \begin{pspicture}(-3,-.5)(3,1.5)
                \psset{unit=.4cm}
                \psline(-4,0)(-1,3)(0,2)(2,2)(4,0)
                \psdots(-4,0)(-3,1)(-2,2)(-1,3)(0,2)(2,2)(3,1)(4,0)
                \psline(-2,0)(-1,1)(1,1)(2,0)
                \psdots(-2,0)(-1,1)(1,1)(2,0)
                \psline(0,0)
                \psdots(0,0)
                \psline[linestyle=dashed,dash=2pt 2pt](-6,0)(6,0)
                \rput(-4,1){$\pi_2$}
                \rput(-2,1){$\pi_1$}
                \rput(0,.5){$\pi_0$}
                \rput(-4,-.8){-4}
                \rput(-2,-.8){-2}
                \rput(0,-.8){0}
                \rput(2,-.8){2}
                \rput(4,-.8){4}
            \end{pspicture}
           }
            \caption{\small (a) A triple ($\pi_1$, $\pi_2$, $\pi_3$) $\in \Pi_3^{(1)}(t)$ of weight $t^4$.
            (b) A triple ($\pi_0$, $\pi_1$, $\pi_2$) $\in \Pi_3^{(0)}(t)$ of weight $t^2$.}
            \label{fig:H3}
    \end{figure}
\item
  For $0 \leq i \leq n$, let $\Pi_{n, \overline{i}}^{(k)}(t)$
(resp. $\Omega_{n, \overline{i}}^{(k)}(t)$) be the set of
$n$-tuples $(\pi_k, \pi_{k+1}, \pi_{k+2},\cdots, \pi_{k+n-1})$
of $t$-large (resp. $t$-small) Schr\"{o}der paths subject to the
following three conditions (See Figure~\ref{fig:H32}):
\begin{itemize}
  \item[$\bullet$] $\pi_{k+j}$ goes from $(-k-2j, 0)$ to $(k+2j, 0)$, for $0 \leq j \leq i-1$.
  \item[$\bullet$] $\pi_{k+j}$ goes from $(-k-2j, 0)$ to $(k+2j+2, 0)$, for $i \leq j \leq n-1$.
  \item[$\bullet$] Any two paths do not intersect.
\end{itemize}
    \begin{figure}[ht]
        \centering
        \subfigure[]{
            \begin{pspicture}(-3,-.5)(4,1.5)
                \psset{unit=.4cm}
                \psline(-5,0)(-2,3)(4,3)(7,0)
                \psdots(-5,0)(-4,1)(-3,2)(-2,3)(0,3)(2,3)(4,3)(5,2)(6,1)(7,0)
                \psline(-3,0)(-1,2)(0,1)(2,1)(3,0)
                \psdots(-3,0)(-2,1)(-1,2)(0,1)(2,1)(3,0)
                \psline(-1,0)(1,0)
                \psdots(-1,0)(1,0)
                \psline[linestyle=dashed,dash=2pt 2pt](-6,0)(8,0)
                \rput(-5,1){$\pi_3$}
                \rput(-3,1){$\pi_2$}
                \rput(-1,.5){$\pi_1$}
                \rput(-5,-.8){-5}
                \rput(-3,-.8){-3}
                \rput(-1,-.8){-1}
                \rput(1,-.8){1}
                \rput(3,-.8){3}
                \rput(5,-.8){5}
                \rput(7,-.8){7}
            \end{pspicture}
        }
        \subfigure[]{
            \begin{pspicture}(-3,-.5)(3.5,1.5)
                \psset{unit=.4cm}
                \psline(-4,0)(-1,3)(0,2)(2,2)(3,1)(5,1)(6,0)
                \psdots(-4,0)(-3,1)(-2,2)(-1,3)(0,2)(2,2)(3,1)(5,1)(6,0)
                \psline(-2,0)(-1,1)(1,1)(2,0)
                \psdots(-2,0)(-1,1)(1,1)(2,0)
                \psline(0,0)
                \psdots(0,0)
                \psline[linestyle=dashed,dash=2pt 2pt](-5,0)(7,0)
                \rput(-4,1){$\pi_2$}
                \rput(-2,1){$\pi_1$}
                \rput(0,.5){$\pi_0$}
                \rput(-4,-.8){-4}
                \rput(-2,-.8){-2}
                \rput(0,-.8){0}
                \rput(2,-.8){2}
                \rput(4,-.8){4}
                \rput(6,-.8){6}
            \end{pspicture}
        }
      \caption{\small (a) A triple ($\pi_1$, $\pi_2$, $\pi_3$) $\in \Pi_{3, \overline2}^{(1)}(t)$ of weight $t^5$.
      (b) A triple ($\pi_0$, $\pi_1$, $\pi_2$) $\in \Pi_{3, \overline2}^{(0)}(t)$ of weight $t^3$.}
      \label{fig:H32}
    \end{figure}
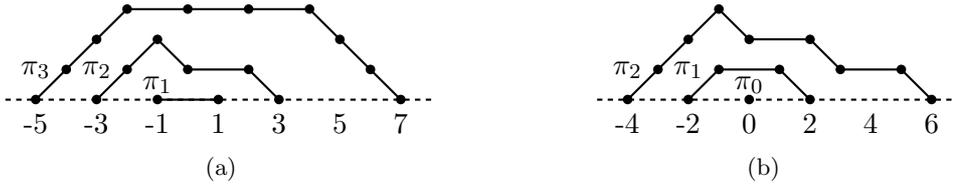
\end {enumerate}

The weight of a $n$-tuple is the product of the weights of all
the component paths; and the weight of a set $X$ of $n$-tuples,
denoted by $|X|$, is the sum of weights of all $n$-tuples in
this set.


\subsection{Lindstr\"om-Gessel-Viennot lemma}
 A family $(p_1,p_2,\dots p_n)$ of lattices paths $p_i$, $1\le
i\le n$, is called \emph{non-intersecting} if no two paths in
the family have a common point. The Lindstr\"om-Gessel-Viennot
lemma associates determinants with non-intersecting path
families in an acyclic directed graph with weights on its edges.
The following simplified version serves our needs:

\begin{lem}[\textbf{Lindstr\"om-Gessel-Viennot}]
Consider the graph $G$. Let $X_1,X_2,\dots ,X_n$ and
$Y_1,Y_2,\dots, Y_n$ be lattice points on the $x$-axis. Then the
total weight of all families $(p_1,p_2,\dots, p_n)$ of
non-intersecting $t$-Schr\"oder paths, $p_i$ running from $X_i$
to $Y_i$, is given by the determinant
$$\det_{1\le i,j\le n} (a_{i,j}),$$
where $a_{i,j}$ is the total weight of lattice paths from $X_i$ to
$Y_j$.
\end{lem}

From the Lindstr\"{o}m-Gessel-Viennot lemma and the models
$\Pi_{n}^{(k)}(t)$, $\Omega_{n}^{(k)}(t)$, we immediately have
the following.

\begin{lem} For integers $n,k\ge 0$, we have
\begin{equation*}
\left|\Pi_{n}^{(k)}(t)\right|=\det(H_{n}^{(k)}(t)) \quad \mbox{and}
\quad \left|\Omega_{n}^{(k)}(t)\right|=\det(G_{n}^{(k)}(t)).
\end{equation*}
\end{lem}


Similarly, from the models $\Pi_{n,\overline{i}}^{(k)}(t)$,
$\Omega_{n,\overline{i}}^{(k)}(t)$ and the definitions of $H_{n,
\overline{i}}^{(k)}(t)$, $G_{n, \overline{i}}^{(k)}(t)$
 we have the following.
\begin{lem}\label{0}
For integers $n,k\ge 0$, we have
  \begin{enumerate}
     \item $\left|\Pi_{n, \overline{0}}^{(k)}(t)\right|=\det(H_{n,\overline{0}}^{(k)}(t))=\det(H_n^{(k+1)}(t))=\left|\Pi_n^{(k+1)}(t)\right|$.\\
     \item $\left|\Omega_{n,\overline{0}}^{(k)}(t)\right|=\det(G_{n,\overline{0}}^{(k)}(t))=\det(G_n^{(k+1)}(t))=\left|\Omega_n^{(k+1)}(t)\right|$.\\
     \item $\left|\Pi_{n,\overline{n}}^{(k)}(t)\right|=\det(H_{n,\overline{n}}^{(k)}(t))=\det(H_n^{(k)}(t))=\left|\Pi_n^{(k)}(t)\right|$.\\
     \item $\left|\Omega_{n, \overline{n}}^{(k)}(t)\right|=\det(G_{n,\overline{n}}^{(k)}(t))=\det(G_n^{(k)}(t))=\left|\Omega_n^{(k)}(t)\right|$.
     \item $\left|\Pi_{n,\overline{i}}^{(k)}(t)\right|= \det(H_{n,\overline{i}}^{(k)}(t))$, for $1\le i\le n-1$.
     \item $\left|\Omega_{n,\overline{i}}^{(k)}(t)\right|=\det(G_{n,\overline{i}}^{(k)}(t))$, for $1\le i\le n-1$.
  \end{enumerate}
\end{lem}

\section{Two Key lemmas}\label{se:KL}
Our proof of the main results bases on two key lemmas, which we
introduce in this section. Before that we need another easy fact
of which we omit the proof.

\begin{lem}\label{A}
We have
$$\det(A_n^{(k, k+1)})=\sum_{i=0}^{n}\alpha^{i}\beta^{n-i}\det(A^{(k)}_{n, \overline{i}}).$$
\end{lem}

The first key lemma relates certain tuples of $t$-large
Schr\"oder paths with the determinants of certain $t$-small
Schr\"oder numbers. Let $\Pi_{n, \overline i}^*(t)$ be
 the set of $n$-tuples of
$t$-large Schr\"{o}der paths in $\Pi_{n, \overline i}^{(1)}(t)$
in which none of its paths touches the point ($2i+1$, $0$). See
Fig.~\ref{fig:H32**}(a) for example.

\begin{lem}[\textbf{Key Lemma I}]\label{c}
For $1 \leq i \leq n$, we have
\begin{equation*}
\left|\Pi_{n, \overline i}^*(t) \right|=(1+t)^{n} \det( G_{n,
\overline i}^{(0)}(t) ).
\end{equation*}
\end{lem}
    \begin{proof}
        We count in two parts. Let $X$ (resp. $Y$) be the set of
        $n$-tuples in $\Pi_{n, \overline i}^*(t)$ with
        $\pi_1=\textsf{L}$ (resp. $\pi_1 =\textsf{UD}$). Note
        that
        $\left|\Pi_{n, \overline i}^*(t)\right| = |X| + |Y|.$\\

$\bullet$ For $X$: There is a bijection between $X$ and
        $\Pi_{n-1, \overline{i-1}}^{(2)}(t)$ by mapping
        $(\textsf{L}, \pi_2, \pi_3, \ldots, \pi_n) \in X$ to
        $(\pi_2', \pi_3', \ldots, \pi_n') \in \Pi_{n-1,
        \overline{i-1}}^{(2)}(t)$ where $\pi_j = \textsf{U}
        \pi_j' \textsf{D}$ for $2 \leq j\leq n$. See
        Figure~\ref{fig:H32**} as an example.

        \begin{figure}[ht]
        \centering
        \subfigure[]{
            \begin{pspicture}(-3,-.5)(4,1.5)
                \psset{unit=.4cm}
                \psline(-5,0)(-2,3)(4,3)(7,0)
                \psdots(-5,0)(-4,1)(-3,2)(-2,3)(0,3)(2,3)(4,3)(5,2)(6,1)(7,0)
                \psline(-3,0)(-1,2)(0,1)(2,1)(3,0)
                \psdots(-3,0)(-2,1)(-1,2)(0,1)(2,1)(3,0)
                \psline(-1,0)(1,0)
                \psdots(-1,0)(1,0)
                \psline[linestyle=dashed,dash=2pt 2pt](-6,0)(8,0)
                \rput(-5,1){$\pi_3$}
                \rput(-3,1){$\pi_2$}
                \rput(-1,.5){$\pi_1$}
                \rput(-5,-.8){-5}
                \rput(-3,-.8){-3}
                \rput(-1,-.8){-1}
                \rput(1,-.8){1}
                \rput(3,-.8){3}
                \rput(5,-.8){5}
                \rput(7,-.8){7}
                \psdots[dotstyle=x, dotscale=3](5,0)
            \end{pspicture}
        }
        \subfigure[]{
            \begin{pspicture}(-3,-.5)(3.5,1)
                \psset{unit=.4cm}
                \psline(-4,0)(-3,1)(-2,2)(0,2)(2,2)(4,2)(5,1)(6,0)
                \psdots(-4,0)(-3,1)(-2,2)(0,2)(2,2)(4,2)(5,1)(6,0)
                \psline(-2,0)(-1,1)(0,0)(2,0)
                \psdots(-2,0)(-1,1)(0,0)(2,0)
                \psline[linestyle=dashed,dash=2pt 2pt](-6,0)(7,0)
                \rput(-4,1){$\pi_3'$}
                \rput(-2,1){$\pi_2'$}
                \rput(-6,-.8){-6}
                \rput(-4,-.8){-4}
                \rput(-2,-.8){-2}
                \rput(0,-.8){0}
                \rput(2,-.8){2}
                \rput(4,-.8){4}
                \rput(6,-.8){6}
            \end{pspicture}
            }
            \caption{\small A bijection between (a) A $3$-tuple ($\pi_1$, $\pi_2$, $\pi_3$) $\in X \subset \Pi_{3, \overline2}^{*}(t)$ of weight
            $t^5$ and
            (b) A $2$-tuple ($\pi_2'$, $\pi_3'$) $\in \Pi_{2, \overline1}^{(2)}(t)$ of weight $t^4$.}
          \label{fig:H32**}
        \end{figure}

Hence the weight of $(\textsf{L}, \pi_2, \pi_3, \ldots,
        \pi_n) \in X$ is equal to $t$ times the weight of
        $(\pi_2', \pi_3', \ldots, \pi_n') \in \Pi_{n-1,
        \overline{i-1}}^{(2)}(t)$. Therefore,
$$|X|=  t \left| \Pi_{n-1, \overline{i-1}}^{(2)}(t)\right|
= \det\left(
                             \begin{array}{c|ccccccc}
                               t & r_1(t) & r_2(t) & \cdots & r_{i-1}(t) & r_{i+1}(t) & \cdots & r_{n}(t)\\ \hline
                               0 & r_2(t) & r_3(t) & \cdots & r_{i}(t) & r_{i+2}(t) & \cdots & r_{n+1}(t) \\
                               \vdots & \vdots & \vdots & \ddots & \vdots & \vdots & \ddots & \vdots \\
                               0 & r_n(t) & r_{n+1}(t) & \cdots & r_{n+i-2}(t) & r_{n+i}(t) & \cdots & r_{2n-1}(t) \\
                             \end{array}
                        \right)_{n\times n}.
$$

$\bullet$ For $Y$: There is a weight-invariant bijection between
        $Y$ and $\Pi_{n, \overline i}^{(0)}(t)$, which carries
        $(\textsf{UD}, \pi_2, \pi_3, \ldots, \pi_n) \in Y$ to
        $(\pi_0',\pi_1',\ldots,\pi_{n-1}') \in \Pi_{n, \overline
        i}^{(0)}(t)$ where
        $\pi_i=\textsf{U}\pi_{i-1}'\textsf{D}$, $1 \leq i \leq
        n$. See Figure~\ref{fig:Y} as an example. Hence

        \begin{figure}[h]
        \subfigure[]{
            \begin{pspicture}(-3,-.5)(3,1)
                \psset{unit=.4cm}
                \psline(-5,0)(-4,1)(-3,2)(-2,3)(0,3)(2,3)(3,2)(4,1)(6,1)(7,0)
                \psdots(-5,0)(-4,1)(-3,2)(-2,3)(0,3)(2,3)(3,2)(4,1)(6,1)(7,0)
                \psline(-3,0)(-2,1)(-1,2)(1,2)(2,1)(3,0)
                \psdots(-3,0)(-2,1)(-1,2)(1,2)(2,1)(3,0)
                \psline(-1,0)(0,1)(1,0)
                \psdots(-1,0)(0,1)(1,0)
                \psline[linestyle=dashed,dash=2pt 2pt](-6,0)(8,0)
                \rput(-5,1){$\pi_3$}
                \rput(-3,1){$\pi_2$}
                \rput(-1,.5){$\pi_1$}
                \rput(-5,-.8){-5}
                \rput(-3,-.8){-3}
                \rput(-1,-.8){-1}
                \rput(1,-.8){1}
                \rput(3,-.8){3}
                \rput(5,-.8){5}
                \rput(7,-.8){7}
                \psdots[dotstyle=x, dotscale=3](5,0)
            \end{pspicture}
           }
           \hspace{1cm}
        \subfigure[]{
            \begin{pspicture}(-3,-.5)(3,1)
                \psset{unit=.4cm}
                \psline(-4,0)(-3,1)(-2,2)(0,2)(2,2)(3,1)(4,0)(6,0)
                \psdots(-4,0)(-3,1)(-2,2)(0,2)(2,2)(3,1)(4,0)(6,0)
                \psline(-2,0)(-1,1)(1,1)(2,0)
                \psdots(-2,0)(-1,1)(1,1)(2,0)
                \psdots(0,0)
                \psline[linestyle=dashed,dash=2pt 2pt](-6,0)(7,0)
                \rput(-4,1){$\pi_2'$}
                \rput(-2,1){$\pi_1'$}
                \rput(0,.5){$\pi_0'$}
                \rput(-4,-.8){-4}
                \rput(-2,-.8){-2}
                \rput(0,-.8){0}
                \rput(2,-.8){2}
                \rput(4,-.8){4}
                \rput(6,-.8){6}
            \end{pspicture}
            }
            \caption{\small A bijection between (a) A $3$-tuple ($\pi_1$, $\pi_2$, $\pi_3$) $\in Y \subseteq \Pi_{3, \overline2}^{*}(t)$ of weight $t^4$.
            and (b) A $3$-tuple ($\pi_0'$, $\pi_1'$, $\pi_2'$) $\in \Pi_{3, \overline2}^{(0)}(t)$ of weight $t^4$.}
            \label{fig:Y}
        \end{figure}
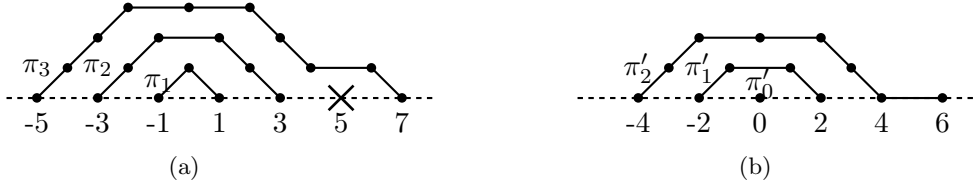

            $$|Y|=\left|  \Pi_{n, \overline i}^{(0)}(t) \right| = \det\left(
                             \begin{array}{ccccccc}
                                r_0(t) & r_1(t) & \cdots & r_{i-1}(t) & r_{i+1}(t) & \cdots & r_{n}(t)\\
                                r_1(t) & r_2(t) & \cdots & r_{i}(t) & r_{i+2}(t) & \cdots & r_{n+1}(t) \\
                                \vdots & \vdots & \ddots & \vdots & \vdots & \ddots & \vdots \\
                                r_{n-1}(t) & r_{n}(t) & \cdots & r_{n+i-2}(t) & r_{n+i-1}(t) & \cdots & r_{2n-1}(t) \\
                             \end{array}
                        \right)_{n\times n}$$

Now by the fact that $r_0(t)=1$, $s_0(t)=1$,
        $r_n(t)=(1+t)s_n(t)$ for $n \ge 1$  and direct
        calculation we have
$$\left|\Pi_{n, \overline i}^*(t)\right| = |X| + |Y|= (1+t)^n \det(  G_{n, \overline i}^{(0)}(t) ),$$
as desired.
%
    \end{proof}


The second key lemma relates determinants of certain $t$-small
Schr\"oder numbers to determinants (of smaller size) of certain
$t$-large Schr\"oder numbers.

\begin{lem}[\textbf{Key Lemma II}]\label{4}
For $1 \leq i \leq n$, we have
    $$\det( G_{n,\overline i}^{(0)}(t) )
    = \det( H_{n-1,\overline{i-1}}^{(1)}(t) ) + (1+t)^{n-1} \det( G_{n-1,\overline i}^{(0)}(t) ).$$
\end{lem}
    \begin{proof}
    Applying Key Lemma I on $\det( G_{n-1,\overline i}^{(0)}(t)
    )$ and (5),(6) of Lemma \ref{0} it suffices to prove

$$\left| \Omega_{n,\overline i}^{(0)}(t) \right| = \left| \Pi_{n-1, \overline {i-1}}^{(1)}(t) \right|
            + \left| \Pi_{n-1, \overline i}^*(t) \right|.$$

Again we count in two parts.
Let $X$ be the set of $n$-tuples
            $(\omega_0,\omega_1,\ldots,\omega_{n-1})$ in
            $\Omega_{n,\overline i}^{(0)}(t)$ subject to the
            conditions that, for $\omega_i$, the downstep begins
            at $(2i+1,1)$ is the first down step from $y=1$ to
            $y=0$, and $Y:=\Omega_{n,\overline
            i}^{(0)}(t)\backslash
            X$.\\


$\bullet$ For $X$: There is a weight-invariant bijection $f$
            between $X$ and $\Pi_{n-1, \overline{i-1}}^{(1)}(t)$
            by $f(\omega_0,\omega_1,\ldots,\omega_{n-1}) =
            (\pi_1,\pi_2,\ldots,\pi_{n-1})$ where
            $\omega_j=\textsf{U} \pi_j \textsf{D}$ for $1\leq
            j\leq n-1$. See Figure~\ref{fig:4X} as an example.
            Thus

            \begin{figure}[h]
            \centering
             \subfigure[]{
                \begin{pspicture}(-3,-.5)(4,2)
                    \psset{unit=.4cm}
                    \psline(-6,0)(-5,1)(-4,2)(-3,3)(-2,4)(0,4)(1,5)(2,4)(4,4)(5,3)(6,2)(7,1)(8,0)
                    \psdots(-6,0)(-5,1)(-4,2)(-3,3)(-2,4)(0,4)(1,5)(2,4)(4,4)(5,3)(6,2)(7,1)(8,0)
                    \psline(-4,0)(-3,1)(-2,2)(-1,3)(1,3)(2,2)(3,1)(5,1)(6,0)
                    \psdots(-4,0)(-3,1)(-2,2)(-1,3)(1,3)(2,2)(3,1)(5,1)(6,0)
                    \psline(-2,0)(-1,1)(0,2)(1,1)(2,0)
                    \psdots(-2,0)(-1,1)(0,2)(1,1)(2,0)
                    \psdots(0,0)
                    \psline[linestyle=dashed,dash=2pt 2pt](-6.5,0)(8.5,0)
                    \rput(-6,1){$\omega_3$}
                    \rput(-4,1){$\omega_2$}
                    \rput(-2,1){$\omega_1$}
                    \rput(0,.5){$\omega_0$}
                    \rput(-6,-.8){-6}
                    \rput(-4,-.8){-4}
                    \rput(-2,-.8){-2}
                    \rput(0,-.8){0}
                    \rput(2,-.8){2}
                    \rput(4,-.8){4}
                    \rput(6,-.8){6}
                    \rput(8,-.8){8}
                \end{pspicture}
                }
             \subfigure[]{
                \begin{pspicture}(-3,-.5)(4,2)
                    \psset{unit=.4cm}
                    \psline(-5,0)(-4,1)(-3,2)(-2,3)(0,3)(1,4)(2,3)(4,3)(5,2)(6,1)(7,0)
                    \psdots(-5,0)(-4,1)(-3,2)(-2,3)(0,3)(1,4)(2,3)(4,3)(5,2)(6,1)(7,0)
                    \psline(-3,0)(-2,1)(-1,2)(1,2)(2,1)(3,0)(5,0)
                    \psdots(-3,0)(-2,1)(-1,2)(1,2)(2,1)(3,0)(5,0)
                    \psline(-1,0)(0,1)(1,0)
                    \psdots(-1,0)(0,1)(1,0)
                    \psline[linestyle=dashed,dash=2pt 2pt](-6,0)(8,0)
                    \rput(-5,1){$\pi_3$}
                    \rput(-3,1){$\pi_2$}
                    \rput(-1,.5){$\pi_1$}
                    \rput(-5,-.8){-5}
                    \rput(-3,-.8){-3}
                    \rput(-1,-.8){-1}
                    \rput(1,-.8){1}
                    \rput(3,-.8){3}
                    \rput(5,-.8){5}
                    \rput(7,-.8){7}
                \end{pspicture}
                }
                \caption{\small A bijection between (a) A $4$-tuple ($\omega_0$, $\omega_1$, $\omega_2$, $\omega_3$) $\in X \subseteq \Omega_{4, \overline 2}^{(0)}(t)$ of weight $t^4$.
                and (b) A $3$-tuple ($\pi_1$, $\pi_2$, $\pi_3$) $\in \Pi_{3, \overline1}^{(1)}(t)$ of weight $t^4$.}
                \label{fig:4X}
            \end{figure}

            $$|X| = \left| \Pi_{n-1, \overline
            {i-1}}^{(1)}(t) \right|.$$

$\bullet$ For $Y$: For $(\omega_0,\omega_1,\ldots,\omega_{n-1})
            \in Y$, $\omega_i = \textsf{U} \omega_i'
            \textsf{DUD}$ for some $t$-large Schr\"{o}der path
            $\omega_i'$ above $y=1$. Thus there is a
            weight-invariant bijection $g$ between $Y$ and
            $\Pi_{n-1, \overline{i}}^{*}(t)$ denoted by
            $g(\omega_0,\omega_1,\ldots,\omega_{n-1}) =
            (\pi_1,\pi_2,\ldots,\pi_{n-1})$ where $\omega_i =
            \textsf{U} \pi_i \textsf{DUD}$ and $\omega_j =
            \textsf{U} \pi_j \textsf{D}$ for $1\leq j \leq n-1,
            \, j\neq i$. See Figure~\ref{fig:4Y} as an example.
            Hence

            \begin{figure}[h]
            \centering
             \subfigure[]{
                \begin{pspicture}(-3,-.5)(4,2)
                    \psset{unit=.4cm}
                    \psline(-6,0)(-5,1)(-4,2)(-3,3)(-2,4)(0,4)(1,5)(2,4)(4,4)(5,3)(6,2)(7,1)(8,0)
                    \psdots(-6,0)(-5,1)(-4,2)(-3,3)(-2,4)(0,4)(1,5)(2,4)(4,4)(5,3)(6,2)(7,1)(8,0)
                    \psline(-4,0)(-3,1)(-2,2)(-1,3)(1,3)(2,2)(3,1)(4,0)(5,1)(6,0)
                    \psdots(-4,0)(-3,1)(-2,2)(-1,3)(1,3)(2,2)(3,1)(4,0)(5,1)(6,0)
                    \psline(-2,0)(-1,1)(0,2)(1,1)(2,0)
                    \psdots(-2,0)(-1,1)(0,2)(1,1)(2,0)
                    \psdots(0,0)
                    \psline[linestyle=dashed,dash=2pt 2pt](-6.5,0)(8.5,0)
                    \rput(-6,1){$\omega_3$}
                    \rput(-4,1){$\omega_2$}
                    \rput(-2,1){$\omega_1$}
                    \rput(0,.5){$\omega_0$}
                    \rput(-6,-.8){-6}
                    \rput(-4,-.8){-4}
                    \rput(-2,-.8){-2}
                    \rput(0,-.8){0}
                    \rput(2,-.8){2}
                    \rput(4,-.8){4}
                    \rput(6,-.8){6}
                    \rput(8,-.8){8}
                \end{pspicture}
             }
             \subfigure[]{
                \begin{pspicture}(-3,-.5)(4,2)
                    \psset{unit=.4cm}
                    \psline(-5,0)(-4,1)(-3,2)(-2,3)(0,3)(1,4)(2,3)(4,3)(5,2)(6,1)(7,0)
                    \psdots(-5,0)(-4,1)(-3,2)(-2,3)(0,3)(1,4)(2,3)(4,3)(5,2)(6,1)(7,0)
                    \psline(-3,0)(-2,1)(-1,2)(1,2)(2,1)(3,0)
                    \psdots(-3,0)(-2,1)(-1,2)(1,2)(2,1)(3,0)
                    \psline(-1,0)(0,1)(1,0)
                    \psdots(-1,0)(0,1)(1,0)
                    \psline[linestyle=dashed,dash=2pt 2pt](-6,0)(8,0)
                    \rput(-5,1){$\pi_3$}
                    \rput(-3,1){$\pi_2$}
                    \rput(-1,.5){$\pi_1$}
                    \rput(-5,-.8){-5}
                    \rput(-3,-.8){-3}
                    \rput(-1,-.8){-1}
                    \rput(1,-.8){1}
                    \rput(3,-.8){3}
                    \rput(5,-.8){5}
                    \rput(7,-.8){7}
                    \psdots[dotstyle=x, dotscale=3](5,0)
                \end{pspicture}
             }
             \caption{ A bijection between (a) A $4$-tuple ($\omega_0$, $\omega_1$, $\omega_2$, $\omega_3$)
             $\in Y \subseteq \Omega_{4, \overline 2}^{(0)}(t)$ of weight $t^3$.
             and (b) A $3$-tuple ($\pi_1$, $\pi_2$, $\pi_3$) $\in \Pi_{3, \overline2}^{*}(t)$ of weight $t^3$.}
             \label{fig:4Y}
             \end{figure}

            $$|Y| = \left| \Pi_{n-1, \overline{i}}^{*}(t) \right|,$$
and the lemma is proved.
    \end{proof}


\section{Evaluations of $\det(H_{n, \overline{i}}^{(1)}(t))$ and
$\det(H_{n, \overline{i}}^{(0)}(t))$ }\label{sec4}

In this section we use two Key Lemmas to derive recurrence
formulae for $\det(H_{n, \overline{i}}^{(1)}(t))$ and
$\det(H_{n, \overline{i}}^{(0)}(t))$ combinatorially. The
summands in the formulae involve $t$-small Schr\"oder numbers.

\begin{lem}\label{b}
 For $1\leq i\leq n$, we have
 $$\det( H_{n, \overline{i}}^{(1)}(t) )
 = (1+t)^{n} \det( H_{n-1,\overline{i-1}}^{(1)}(t))
 +(1+t)^{n+1}\det(H_{n-1,\overline{i}}^{(1)}(t))
 +(1+t)^{2n-1}\det( G_{n-1, \overline{i}}^{(0)}(t) ).$$
\end{lem}

    \begin{proof}
        Since $\det( H_{n, \overline{i}}^{(1)}(t) ) = (1+t)^{n}
        \det( G_{n, \overline{i}}^{(1)}(t) )$, it suffices to
        prove
        $$\left| \Omega_{n, \overline{i}}^{(1)}(t)\right|
        = \left| \Pi_{n-1, \overline{i-1}}^{(1)}(t)\right| + (1+t) \left| \Pi_{n-1, \overline{i}}^{(1)}(t) \right|
        + (1+t)^{n-1} \left| \Omega_{n-1,\overline{i}}^{(0)}(t)
        \right|.$$

The idea is the same as before by observing the first time
       $\omega_{i+1}$ descending from $y=2$ to $y=1$. Let $X$,
       $Y$ and $Z$ be respectively the subsets of $\Omega_{n,
       \overline i}^{(1)}(t)$ having the property (i), (ii) and
       (iii) for $\omega_{i+1}$.


            \begin{enumerate}
              \item[(i)] $(2i+1,2)\to (2i+2,1)$ is the first $\textsf{D}$ leaving $y=2$.
              \item[(ii)] $(2i-1,2)\to (2i,1)$ is the first $\textsf{D}$ leaving $y=2$, and after that it will never touches $y=2$ again.
              \item[(iii)] $(2i-1,2)\to (2i,1)$ is the first $\textsf{D}$ leaving $y=2$, and immediately followed by a $\textsf{U}$ step ($(2i,1)\to (2i+1,2)$).
            \end{enumerate}

Note that
        $$\left| \Omega_{n, \overline{i}}^{(1)}(t)\right|=|X|+|Y|+|Z|.$$

$\bullet$ For $X$: Let $f$ be the weight-invariant bijection between
$X$ and
        $\Pi_{n-1, \overline {i-1}}^{(1)}(t)$ by
        $f(\omega_1,\omega_2,\ldots,\omega_n) =
        (\pi_1,\pi_2,\ldots,\pi_{n-1})$ where $\omega_i =
        \textsf{UU}\pi_{i-1} \textsf{DD}$, $2\leq i\leq n$
        (Note that $\omega_1=\textsf{UD}$). Hence $$|X| = \left|
        \Pi_{n-1, \overline{i-1}}^{(1)}(t) \right|.$$

$\bullet$ For $Y$: For an $n$-tuple
        $(\omega_1,\omega_2,\ldots,\omega_n)$ in $Y$, there
        exist $(n-1)$-tuples $(\pi_1,\pi_2,\ldots,\pi_{n-1})$ in
        $\Pi_{n-1, \overline{i}}^{(1)}(t)$ satisfying that
        $\omega_j = \textsf{UU}\pi_{j-1} \textsf{DD}$ where
        $2\leq j\leq n$, $j\neq i+1$, and $\omega_{i+1} =
        \textsf{UU}\pi_{i} \textsf{DLD}$ or $\omega_{i+1} =
        \textsf{UU}\pi_{i} \textsf{DDUD}$.
        Thus an $n$-tuple in $Y$ corresponds
        with two $(n-1)$-tuples in $\Pi_{n-1, \overline
        {i}}^{(1)}(t)$ and then $$|Y| = (1+t) \left| \Pi_{n-1,
        \overline {i}}^{(1)}(t) \right|.$$

$\bullet$ For $Z$: For an $n$-tuple
        $(\omega_1,\omega_2,\ldots,\omega_n)$ in $Z$, the
        steps of $\omega_{i+1}$ starting from $(2i-1,2)$ is
        $\textsf{DUDD}$ and $\omega_{i+1}$ touches $(2i+1,2)$.
        Hence it corresponds with an $(n-1)$-tuple in $\Pi_{n-1,
        \overline i}^{(1)}(t)$ in which none of its paths touches the point $(2i+1,0)$, i.e., an $(n-1)$-tuple in
        $\Pi_{n-1, \overline i}^{*}(t)$. By Lemma \ref{c}, $$|Z|
        = \left| \Pi_{n-1, \overline i}^{*}(t) \right| =
        (1+t)^{n-1} \det( G_{n-1, \overline i}^{(0)}(t) ) =
        (1+t)^{n-1} \left| \Omega_{n-1, \overline i}^{(0)}(t)
        \right|.$$

Hence the proof is complete.
    \end{proof}

\begin{lem}\label{a}
    For $1\leq i\leq n$, we have
    $$\det( H_{n, \overline{i}}^{(0)}(t) )
    = \det( H_{n-1, \overline{i-1}}^{(1)}(t) ) + t\det( H_{n-1, \overline{i}}^{(1)}(t) )
    + (1+t)^{n-1} \det( G_{n-1, \overline{i}}^{(0)}(t) ).$$
\end{lem}
    \begin{proof}
        The idea is the same. For every $n$-tuple $(\pi_0, \pi_1, \ldots, \pi_{n-1})\in \Pi_{n, \overline {i}}^{(0)}(t)$, the end point of $\pi_i$ is $(2i+2,0)$,
        and we look at $\pi_i$ descending from $y=1$ to $y=0$ for the first
time. There are three cases.
 Let $X$, $Y$ and $Z$ be respectively the subsets of $\Pi_{n, \overline {i}}^{(0)}(t)$ having the property (i), (ii) and
       (iii) for $\pi_i$.
            \begin{enumerate}
              \item[(i)] $(2i+1,1)\to (2i+2,0)$ is the first $\textsf{D}$ leaving $y=1$.
              \item[(ii)] $(2i-1,1)\to (2i,0)$ is the first $\textsf{D}$ leaving $y=1$, and immediately followed by a $\textsf{L}$ step ($(2i,0)\to (2i+2,0)$).
              \item[(iii)] $(2i-1,1)\to (2i,0)$ is the first $\textsf{D}$ leaving $y=1$,
              and immediately followed by a $\textsf{U}$ step
              ($(2i,0)\to (2i+1,1)$) and then a $\textsf{D}$
              step ($(2i+1,1)\to (2i+2,0)$).
            \end{enumerate}
        Note that $$\left| \Pi_{n, \overline {i}}^{(0)}(t) \right| = |X| + |Y| + |Z|.$$

$\bullet$ For $X$: It easy to see that an $n$-tuple $(\pi_0, \pi_1,
\ldots, \pi_{n-1})\in X$ corresponds to an $(n-1)$-tuple $(\pi_1',
\pi_2', \ldots, \pi_{n-1}')\in \Pi_{n-1, \overline {i-1}}^{(1)}(t)$
where $\pi_j=$ \textsf{U}$\pi_j'$\textsf{D} for $1 \leq j \leq n-1$.
Hence
$$|X|=\left| \Pi_{n-1, \overline {i-1}}^{(1)}(t) \right|=\det( H_{n-1, \overline {i-1}}^{(1)}(t) )$$

$\bullet$ For $Y$: An $n$-tuple $(\pi_0, \pi_1, \ldots,
\pi_{n-1})\in Y$ corresponds to an $(n-1)$-tuple $(\pi_1', \pi_2',
\ldots, \pi_{n-1}')\in \Pi_{n-1, \overline {i}}^{(1)}(t)$ where
$\pi_j=$ \textsf{U}$\pi_j'$\textsf{DL} for $1 \leq j \leq n-1$.
Hence
$$|Y|=t \left| \Pi_{n-1, \overline {i}}^{(1)}(t) \right|= t \det( H_{n-1, \overline {i}}^{(1)}(t) ) $$

$\bullet$ For $Z$: An $n$-tuple $(\pi_0, \pi_1, \ldots,
\pi_{n-1})\in Z$ corresponds to an $(n-1)$-tuple $(\pi_1', \pi_2',
\ldots, \pi_{n-1}')\in \Pi_{n-1, \overline {i}}^{*}(t)$ where
$\pi_j=$ \textsf{U}$\pi_j'$\textsf{DUD} for $1 \leq j \leq n-1$.
Hence,
$$|Z|=\left| \Pi_{n-1, \overline {i}}^{*}(t) \right|=(1+t)^{n-1} \det( G_{n-1, \overline {i}}^{(0)}(t) ).$$

The proof is complete by combining three identities.

    \end{proof}


\section{Two recurrences}\label{sec5}

The goal of this section is to derive a recurrence formula for
$\det( H_{n, \overline{i}}^{(0)}(t))$ (resp. $\det( H_{n,
\overline{i}}^{(1)}(t))$), which will involve only the $t$-large
(resp. $t$-small) Schr\"oder numbers.

For simplicity, for $0 \leq i \leq n$, let
     \begin{align*}
       P_{n, \overline{i}}(t) &= (1+t)^{-\binom{n}{2}} \det( H_{n, \overline{i}}^{(0)}(t)),\\
       Q_{n, \overline{i}}(t) &= (1+t)^{-\binom{n+1}{2}} \det( H_{n, \overline{i}}^{(1)}(t)),\\
       R_{n, \overline{i}}(t) &= (1+t)^{-\binom{n}{2}} \det( G_{n, \overline{i}}^{(0)}(t)),
     \end{align*}
with $P_{0, \overline{0}}=1$ and $P_{i, \overline{j}}=0$ if $j> i$
(similar initial conditions hold for $Q$ and $R$). The following are
the direct translations of lemmas~\ref{4},~\ref{b}, and~\ref{a}.

 \begin{lem}\label{4'}  For $1 \leq i \leq n$, we have
 \begin{enumerate}
  \item $R_{n, \overline{i}}(t)=Q_{n-1,
  \overline{i-1}}(t)+R_{n-1, \overline{i}}(t).$
  \item $Q_{n, \overline{i}}(t)=Q_{n-1,
  \overline{i-1}}(t)+(1+t)Q_{n-1, \overline{i}}(t)+R_{n-1,
  \overline{i}}(t).$
  \item $P_{n,
  \overline{i}}(t)=Q_{n-1, \overline{i-1}}(t)+tQ_{n-1,
  \overline{i}}(t)+R_{n-1, \overline{i}}(t).$
  \end{enumerate}
 \end{lem}

We first deal with the cases $i=0$ and $i=n$.

 \begin{lem}\label{i=0,n} We have
  \begin{enumerate}
    \item[(i)] $Q_{n, \overline{0}}(t)=1+(1+t)Q_{n-1, \overline{0}}(t)$, and $Q_{n,\overline{n}}(t)=1$.
    \item[(ii)] $P_{n, \overline{0}}(t)=(1+t)P_{n-1, \overline{0}}(t)$, and $P_{n,\overline{n}}(t)=1$.
    \item[(iii)] $R_{n, \overline{0}}(t)=R_{n,\overline{n}}(t)=1$.
  \end{enumerate}
 \end{lem}
 \begin{proof}
(i) We have $Q_{n, \overline{n}}(t)=Q_{n-1,
\overline{n-1}}(t)=\cdots=Q_{0,\overline{0}}(t)=1$ by
Lemma~\ref{4'}(2). By Lemma~\ref{0}, $Q_{n, \overline{0}}(t)=
(1+t)^{-\binom{n+1}{2}} \det( H_{n, \overline{0}}^{(1)}(t)
)=(1+t)^{-\binom{n+1}{2}} \det( H_{n}^{(2)}(t)
 )$.
We use the following identity (from (\ref{dog})) to compute $\det(
H_{n}^{(2)}(t))$:
 $$\det( H_{n+1}^{(0)}(t) )\det( H_{n-1}^{(2)}(t) )=\det( H_{n}^{(0)}(t) ) \det( H_{n}^{(2)}(t) )
  -\det( H_n^{(1)}(t) )^2.$$
By applying the know formulae for $\det(H_n^{(0)}(t)), \det(
H_n^{(1)}(t))$ and some simple calculation we reach at
$$Q_{n,\overline{0}}(t)=Q_{n-1,\overline{0}}(t)+(1+t)^n.$$
It can be solved that $Q_{n,\overline{0}}(t)=\sum_{k=0}^{n}(1+t)^k$,
therefore $$Q_{n,\overline{0}}(t)=1+(1+t)Q_{n-1, \overline{0}}(t).$$

(ii) We have $P_{n, \overline{n}}(t)=Q_{n-1, \overline{n-1}}(t)=1$
by Lemma~\ref{4'}(3) and (i). By Lemma~\ref{0}, we have $\det(H_{n,
\overline{0}}^{(0)}(t))=\det(H_{n}^{(1)}(t))=\det(H_{n,
 \overline{n}}^{(1)}(t))$.
 Thus $$P_{n, \overline{0}}(t)=(1+t)^{-\binom{n}{2}}\det(H_{n, \overline{0}}^{(0)}(t))=(1+t)^{-\binom{n}{2}}\det(H_{n,
 \overline{n}}^{(1)}(t))=(1+t)^nQ_{n,\overline{n}}(t)=(1+t)^n.$$
  Hence $$P_{n, \overline{0}}(t)=(1+t)P_{n-1, \overline{0}}(t).$$

(iii) We have $R_{n, \overline{n}}(t)=Q_{n-1, \overline{n-1}}(t)=1$
by Lemma~\ref{4'}(1) and (i).
 Besides, $$R_{n, \overline{0}}(t)=(1+t)^{-\binom{n}{2}} \det( G_{n,\overline{0}}^{(0)}(t))=(1+t)^{-\binom{n}{2}} \det(
 G_{n}^{(1)}(t))= (1+t)^{-\binom{n+1}{2}} \det(
 H_{n}^{(1)}(t)).$$
 The last identity is from the fact $s_n(t)=(1+t)^{-1}r_n(t)$ for $n \geq 1$. Thus
 $$R_{n, \overline{0}}(t)=(1+t)^{-\binom{n+1}{2}} \det(
 H_{n}^{(1)}(t))=(1+t)^{-\binom{n+1}{2}} \det( H_{n,\overline{n}}^{(1)}(t))=Q_{n, \overline{n}}(t)=1,$$ as desired.
 \end{proof}



The last pieces we need are the recurrence formulaes for $P_{n,
\overline{i}}(t)$ and $Q_{n, \overline{i}}(t)$ for $1\le i\le n$.


 \begin{lem}\label{2}
   For $1 \leq i \leq n$, we have
    $$Q_{n, \overline{i}}(t)=(1+t)Q_{n-1, \overline{i}}(t)+\sum_{k=i}^{n}Q_{k-1, \overline{i-1}}(t).$$
 \end{lem}
 \begin{proof}
 Repeatedly applying Lemma~\ref{4'}(1), we have
$$  R_{n, \overline{i}}(t) =\sum_{k=i}^{n-1}Q_{k,\overline{i-1}}(t) +
R_{i, \overline{i}}(t). $$
%
Since $R_{i, \overline{i}}(t)=Q_{i-1, \overline{i-1}}(t)=1$, we then
have
  $$R_{n, \overline{i}}(t)= \sum_{k=i}^{n}Q_{k-1,\overline{i-1}}(t).$$
 Plug it into Lemma~\ref{4'}(2) and the lemma is proved.
 \end{proof}


 \begin{lem}\label{3}
  For $1 \leq i \leq n$,
  $$ P_{n, \overline{i}}(t)=(1+t)P_{n-1, \overline{i}}(t)+\sum_{k=i}^{n}P_{k-1, \overline{i-1}}(t)$$
 \end{lem}
 \begin{proof} We have
    $$P_{n, \overline{i}}(t) = Q_{n-1, \overline{i-1}}(t) + tQ_{n-1, \overline{i}}(t) + R_{n-1, \overline{i}}(t)=Q_{n, \overline{i}}(t) - Q_{n-1, \overline{i}}(t)$$
   by Lemma~\ref{4'} (3) and (2). Applying the result of Lemma~\ref{2}, we obtain
    \begin{alignat*}{2}
    P_{n, \overline{i}}(t) &= \left( (1+t)Q_{n-1,\overline{i}}(t)+ \sum_{k=i}^{n}Q_{k-1,\overline{i-1}}(t)\right)
                              -\left( (1+t)Q_{n-2, \overline{i}}(t)+ \sum_{k=i}^{n-1}Q_{k-1,\overline{i-1}}(t)\right)\\
                           &= (1+t) \left( Q_{n-1, \overline{i}}(t) - Q_{n-2, \overline{i}}(t)\right)
                              +\sum_{k=i+1}^{n}\left( Q_{k-1, \overline{i-1}}(t) - Q_{k-2, \overline{i-1}}(t)\right) + Q_{i-1,\overline{i-1}}(t)\\
                           &=(1+t)P_{n-1, \overline{i}}(t)+\sum_{k=i}^{n}P_{k-1, \overline{i-1}}(t),
    \end{alignat*}
as desired.
 \end{proof}


\section{Proof of the Main Theorem}\label{sec6}
We are now ready to prove the main theorem.

\begin{proof}[Proof of Theorem~\ref{th:L01} (1)]
Expanding $\Theta_n(t)$ by Lemma~\ref{A}, we have
$\Theta_n(t)=\sum_{i=0}^{n} \alpha^i \beta^{n-i} P_{n,
\overline{i}}(t)$.
%
Splitting the sum into two parts and applying Lemma~\ref{i=0,n}
and~\ref{3}, we have
    \begin{alignat*}{2}
        \Theta_n(t) &= \beta^n P_{n, \overline{0}}(t)+\sum_{i=1}^{n} \alpha^i \beta^{n-i} P_{n, \overline{i}}(t)\\
                    &= \beta^n (1+t) P_{n-1, \overline{0}}(t)+\sum_{i=1}^{n} \alpha^i \beta^{n-i} \left( (1+t) P_{n-1, \overline{i}}(t) + \sum_{k=i}^{n}P_{k-1, \overline{i-1}}(t) \right) \\
                    &= \beta (1+t) \sum_{i=0}^{n-1} \alpha^i \beta^{n-1-i} P_{n-1, \overline{i}}(t)+\sum_{i=1}^{n} \alpha^i \beta^{n-i} \left( \sum_{k=i}^{n}P_{k-1, \overline{i-1}}(t) \right)\\
                    &= \beta (1+t) \Theta_{n-1}(t) +\sum_{k=1}^{n} \alpha \beta^{n-k} \left( \sum_{i=1}^{k} \alpha^{i-1} \beta^{k-i} P_{k-1, \overline{i-1}}(t) \right)\\
                    &= \beta (1+t)\Theta_{n-1}(t)+\alpha \sum_{m=0}^{n-1} \beta^{m} \Theta_{n-1-m}(t),
    \end{alignat*}
as desired.
\end{proof}

\begin{proof} [Proof of Theorem~\ref{th:L01} (2)]
The proof is similar to above. Expanding $\Phi_n(t)$ by
Lemma~\ref{A}, we have $\Phi_n(t)=\sum_{i=0}^{n} \alpha^i
\beta^{n-i} Q_{n,\overline{i}}(t)$. Splitting the sum in two parts
and using Lemma~\ref{i=0,n} and ~\ref{2} we have
%
%
      \begin{alignat*}{2}
        \Phi_n(t)  &= \beta^n Q_{n, \overline{0}}(t)+\sum_{i=1}^{n} \alpha^i \beta^{n-i} Q_{n, \overline{i}}(t)\\
                   &= \beta^n \left(1+(1+t)Q_{n-1, \overline{0}}(t)\right) +\sum_{i=1}^{n} \alpha^i \beta^{n-i} \left( (1+t) Q_{n-1, \overline{i}}(t) + \sum_{k=i}^{n}Q_{k-1, \overline{i-1}}(t) \right) \\
                   &= \beta^n + \beta (1+t)\Phi_{n-1}(t)+ \sum_{k=1}^{n}\alpha \beta^{n-k} \sum_{i=1}^{k}\alpha^{i-1} \beta^{k-i} Q_{k-1, \overline{i-1}}(t)\\
                   &= \beta^n + \beta (1+t)\Phi_{n-1}(t)+\alpha \sum_{m=0}^{n-1} \beta^{m} \Phi_{n-1-m}(t),
      \end{alignat*}
as desired.
\end{proof}

\begin{proof}[Proof of Theorem~\ref{th:L01} (3)]
 We prove this by induction on $n$. The case $n=1$ holds trivially.
By Lemma~\ref{A} and definition of $R_{n, \overline{i}}(t)$ we can
expand $\Psi_n(t)$ into
$$ \Psi_n(t)= \alpha^n R_{n, \overline{n}}(t)+ \sum_{i=1}^{n-1} \alpha^i\beta^{n-i} R_{n, \overline{i}}(t) + \beta^n R_{n,\overline{0}}(t).$$

%


 Now, by Lemma~\ref{4'}, we have $R_{n, \overline{i}}(t) = Q_{n-1,\overline{i-1}}(t) + R_{n-1,\overline{i}}(t)$ for $1 \leq i \leq n-1$.
 Also $R_{n, \overline{n}}(t)=Q_{n-1, \overline{n-1}}(t)=1$ and $R_{n, \overline{0}}(t)=R_{n-1, \overline{0}}(t)=1$ by Lemma~\ref{i=0,n}.
 Substitute these into above
 we reach
    \begin{alignat*}{2}
        \Psi_n(t) &= \alpha^n Q_{n-1, \overline{n-}1}(t)+ \sum_{i=1}^{n-1} \alpha^i\beta^{n-i} \left(Q_{n-1, \overline{i-1}}(t)+R_{n-1, \overline{i}}(t)\right) + \beta^n R_{n-1, \overline{0}}(t)\\
                  &= \alpha \sum_{j=0}^{n-1} \alpha^j\beta^{n-1-j} Q_{n-1, \overline{j}}(t)+ \beta \sum_{j=0}^{n-1} \alpha^j\beta^{n-1-j} R_{n-1,\overline{j}}(t)\\
                  &= \alpha \Phi_{n-1}(t) + \beta \Psi_{n-1}(t).
   \end{alignat*}
 Then by Theorem~\ref{th:L01} (2) and the induction hypothesis, we get
   \begin{alignat*}{2}
        \Psi_n(t) &= \alpha \Phi_{n-1}(t) + \beta \Psi_{n-1}(t)\\
                  &= \alpha \left( \alpha \sum_{m=0}^{n-2} \beta^{m} \Phi_{n-2-m}(t) + \beta (1+t) \Phi_{n-2}(t) + \beta^{n-1}
                  \right)\\
                  &\phantom{==}+ \beta \left( \alpha \sum_{m=0}^{n-2} \beta^{m} \Psi_{n-2-m}(t) + \beta (1+t) \Psi_{n-2}(t) - t\beta^{n-1}
                  \right)\\
                  &= \alpha \sum_{m=0}^{n-2} \beta^{m} \left( \alpha \Phi_{n-2-m}(t) + \beta \Psi_{n-2-m}(t) \right) + \beta (1+t) \left( \alpha \Phi_{n-2}(t) + \beta \Psi_{n-2}(t) \right)\\
                  &\phantom{==}+ \alpha \beta^{n-1} \Psi_{0}(t) - t
                  \beta^n\\
                  &= \alpha \sum_{m=0}^{n-1} \beta^{m} \Psi_{n-1-m}(t) + \beta (1+t) \Psi_{n-1}(t)- t
                  \beta^n,
   \end{alignat*}
 and the proof is completed.
\end{proof}

\begin{proof}[Proof of Theorem~\ref{th:L01} (4)]
Simply by using the identity $r_n(t)=(1+t)s_n(t)$ for $n \geq 1$ and
we are done.
\end{proof}

\section{Concluding notes} A natural extension is to consider the Hankel determinants
in which each entry is the linear combination of more than two
consecutive terms of $t$-large (or small) Schr\"oder numbers.
However, a proof using lattice path models turns out to be messy and
seems not so attractive. Another natural generalization is to put
different weights with respect to the heights, or consider the
$q$-analogue versions. We leave these interesting problems to the
readers.




\end{document}